\documentclass[]{interact}
\usepackage{fix-cm}
\usepackage{xcolor}
\usepackage{epstopdf}
\usepackage[caption=false]{subfig}

\usepackage[numbers,sort&compress]{natbib}
\bibpunct[, ]{[}{]}{,}{n}{,}{,}
\makeatletter
\def\NAT@def@citea{\def\@citea{\NAT@separator}}
\makeatother

\theoremstyle{plain}
\newtheorem{theorem}{Theorem}[section]
\newtheorem{lemma}[theorem]{Lemma}
\newtheorem{corollary}[theorem]{Corollary}
\newtheorem{proposition}[theorem]{Proposition}

\theoremstyle{definition}
\newtheorem{definition}[theorem]{Definition}
\newtheorem{example}[theorem]{Example}

\theoremstyle{remark}
\newtheorem{remark}[theorem]{Remark}

\newcommand{\N}{\mathbb{N}}

\begin{document}
	
	
	\title{On the Strong Quasiconvexity of Norms and Distance Functions}

		\author{
		\name{Vo Si Trong Long\textsuperscript{a,b} and Nguyen Mau Nam\textsuperscript{c}\thanks{CONTACT Nguyen Mau Nam. Email: mnn3@pdx.edu} \\ \medskip 
        \textbf{Dedicated to Prof. Phan Quoc Khanh on his 80th birthday}}
		\affil{\textsuperscript{a}Faculty of Mathematics and Computer Science, University of Science, Ho Chi Minh City, Vietnam; \textsuperscript{b} Vietnam National University, Ho Chi Minh City, Vietnam;
        \textsuperscript{c}Fariborz Maseeh Department of Mathematics and Statistics,
		Portland State University, Portland, OR 97207, USA.}}

	\maketitle

	\begin{abstract}
    	This paper studies the strong quasiconvexity of norm and distance functions in finite-dimensional normed spaces. Although the Euclidean norm is known to be strongly quasiconvex on bounded convex sets, a complete characterization of this property for general norms remains open. We establish a necessary condition for the strong quasiconvexity of an arbitrary norm function on a convex set. For the $\ell_p$-norm with $1<p\le2$, this condition becomes also sufficient. We then initiate the study of the strong quasiconvexity of distance functions. Our results provide new insights into the geometric properties of norm and distance functions and extend several existing results in the literature.
	\end{abstract}
	

\section{Introduction and Preliminaries}
	
	Convex analysis plays a fundamental role in optimization theory and has been instrumental in a wide range of applications across  economics, engineering, operations research, machine learning, data science, signal processing, control theory, and mathematical finance. For comprehensive treatments and further details on convex analysis and its applications, we refer the reader to \cite{Bauschke,Nesterov,rock27,za28}. Motivated by the need to study optimization problems beyond the classical convex framework, a variety of generalized convexity notions have been introduced and extensively investigated in the literature. Among these notions, the concepts of quasiconvexity and strong quasiconvexity have attracted considerable attention due to their ability to preserve certain important properties of convex functions while allowing for greater flexibility in modeling and applications.

	The notion of strong quasiconvexity was originally introduced by Boris T. Polyak in~\cite{polyak} in the context of optimization, with the aim of ensuring stability and  convergence properties of minimizing sequences. In comparison with quasiconvexity, strong quasiconvexity imposes additional structural conditions that often guarantee the existence and uniqueness of a global optimal solution in many important settings.
	Extensive studies on the properties and characterizations of this class of functions have led to a broad range of applications in both theoretical analysis and numerical methods, including proximal point algorithms, variational inequalities, and equilibrium problems; see \cite{Crouzeix,Grad,jova,Hadjisavvas,Hadjisavvas2,Hadjisavvas1,lara1,lara2,lara3,Nam-Jacob,vanhang,Zalinescu} and the references therein.
	
	Due to the fundamental role played by norm functions and distance functions in convex analysis and related areas, an important and intriguing question is to investigate the strong quasiconvexity properties of these classes of functions.  A landmark result was established by  Jovanović in \cite{jova}, showing that the norm function is strongly quasiconvex on every bounded convex subset of the Euclidean space $\mathbb{R}^n$. Recently, this result was extended by Nam and Sharkansky \cite{Nam-Jacob} to infinite-dimensional spaces, where various conditions ensuring the preservation of strong quasiconvexity were studied.
	Despite these advances, several important questions remain open. In particular, does the strong quasiconvexity of a norm function necessarily imply the boundedness of the underlying convex set? Moreover, under what conditions do distance functions possess strong quasiconvexity properties?
	
	Inspired by the aforementioned works, the main objective of this paper is twofold. 
    First, we show that the strong quasiconvexity of an arbitrary norm function on a convex set implies the boundedness of the underlying set. For the $\ell_p$-norm with $1<p\le2$, we obtain a complete characterization by showing that this property is equivalent to boundedness.
  Second, we establish sufficient conditions for the strong quasiconvexity of distance functions, with particular emphasis on their behavior on strongly convex sets in the sense of Vial \cite{vial}. Through these investigations, we further explore the relationship between the geometric properties of sets and the analytical properties of the associated distance functions.

	The paper is organized as follows. Section \ref{sec3} establishes a characterization of the strong quasiconvexity of norm functions. In Section \ref{sec4} we study the strong quasiconvexity of distance functions relative to balls and strongly convex sets. Several illustrative examples are provided to clarify the theoretical developments. Finally, Section \ref{sec5} concludes the paper with remarks and possible directions for future research.

Throughout this paper, the sets of real numbers and positive integers are denoted by
$\mathbb{R}$ and $\mathbb{N}$, respectively.  The
$\ell_p$-norm on $\mathbb{R}^n$ for $p \geq 1$ is denoted by $\| \cdot \|_p$.  	
For $c \in \mathbb{R}^n$ and $r \geq 0$, we denote the closed $\ell_p$-ball and the
$\ell_p$-sphere with center $c$ and radius $r$ by $\mathbb{B}^p[c; r]$ and
$\mathbb{S}^p[c; r]$, respectively. When $c = 0$, we simply write
$\mathbb{B}^p_r$ and $\mathbb{S}^p_r$ instead of $\mathbb{B}^p[0; r]$ and
$\mathbb{S}^p[0; r]$.

Given an arbitrary norm $\|\cdot\|$ on $\mathbb{R}^n$ and a nonempty set
$\Omega\subset\mathbb{R}^n$, the \textit{distance function} associated with
$\Omega$ and $\|\cdot\|$ is defined by
\[
	d_\Omega(x)=\inf\{\|x-y\|\mid y\in\Omega\}
	\ \; \text{for }x\in\mathbb{R}^n.
\]

We now recall the  definitions of quasiconvexity and strong quasiconvexity used in what follows.
	\begin{definition}\label{defquasiconvex}
Let $C$ be a nonempty convex subset of $\mathbb{R}^n$, and let $\sigma \ge 0$ be a constant. A function $f \colon C \to \overline{\mathbb{R}} = [-\infty, \infty]$ is said to be $\sigma$-\textit{quasiconvex} on $C$ if 
\begin{equation*}\label{eq:quasiconvex_property}
f(\lambda x + (1-\lambda)y) \le \max\{f(x), f(y)\} - \frac{\sigma}{2} \lambda(1-\lambda) \|x - y\|_2^2
\end{equation*}
for all $x, y \in C$ and all $\lambda \in (0,1)$. In the case where $\sigma > 0$, we say that $f$ is $\sigma$-\textit{strongly quasiconvex} on $C$.
\end{definition}

\begin{definition}
	Given a nonempty closed convex subset $C$ of $\mathbb{R}^n$ and a point $x\in C$, the \textit{horizon cone} of $C$ at $x$ is defined by
	\[
	C_{\infty}(x)
	=
	\{v\in\mathbb{R}^n \mid x+tv\in C \text{ for all } t\geq0\}.
	\]
\end{definition}

It is well known that $C_\infty(x)$ is the same for every $x\in C$; see, e.g., \cite[Proposition~6.2]{mordukhovich2023easy}. Thus, it can be simply denoted by $C_\infty$.

We conclude this section by recalling a standard result from convexity; see, e.g.,~\cite{jensen}.
	
	\begin{theorem}\label{thm:midpoint-jensen}
		Let $h\colon [a,b]\to\mathbb{R}$ be a continuous function such that
		\begin{equation}\label{midpointconvex}
			h\!\left(\frac{s+t}{2}\right)
			\;\le\;
			\frac{h(s)+h(t)}{2}\ \; \mbox{\rm for all } s,t\in[a,b].
		\end{equation}
		Then $h$ is convex on $[a,b]$, that is,
		\[
		h\bigl((1-\lambda)s+\lambda t\bigr)
		\;\le\;
		(1-\lambda)h(s) + \lambda h(t)
		\ \; \mbox{\rm for all } s,t\in[a,b],\; \lambda\in[0,1].
		\]
	\end{theorem}
	
Note that a function $h$ satisfying \eqref{midpointconvex} is called \emph{midpoint convex}.

	\section{The $\sigma$-strong quasiconvexity of norm funtions}	\label{sec3}

    In this section, we study necessary conditions for the $\sigma$-strong quasiconvexity of an arbitrary norm function on $\mathbb{R}^n$. We then obtain a necessary and sufficient condition for the $\ell_p$-norm when $1<p\le2$.
     We begin with the following lemma.

	\begin{lemma}\label{lem:1D-midpoint-strong}
		Let $f \colon  [0,1] \to \mathbb{R}$ be a continuous function. 
		Assume that there exists a constant $\alpha \geq 0$ such that		
		\begin{equation}\label{eq:g-midpoint}
			f\!\left(\frac{s+t}{2}\right)
			\le \frac{f(s)+f(t)}{2}
			- \alpha (t-s)^2\; \ \mbox{\rm for all} \ \;  s,t\in[0,1].
		\end{equation}
		Then we have
		\begin{equation*}\label{eq:g-strong}
			f(\lambda)
			\le (1-\lambda)f(0) + \lambda f(1)
			- 4\alpha\,\lambda(1-\lambda)\; \  \mbox{\rm for all }\lambda\in[0,1].
		\end{equation*}
	\end{lemma}
	
	\begin{proof} Given $\lambda\in[0,1],$ define 
		\begin{equation*}g(\lambda) = 4\alpha \lambda^2\ \;\text{and}\ \; 
		h(\lambda) = f(\lambda) - g(\lambda).
        \end{equation*}
		We claim that $h$ is midpoint convex.
		Indeed, fix any $s,t\in [0,1]$ and get
		\[
		\frac{g(s)+g(t)}{2} - g\!\left(\frac{s+t}{2}\right)
		= 4\alpha\Big(\frac{s^2+t^2}{2} - \frac{(s+t)^2}{4}\Big)
		= \alpha (t-s)^2.
		\]
		Combining this with \eqref{eq:g-midpoint}, we obtain
		\[
		h\!\left(\frac{s+t}{2}\right)
		\le \frac{h(s)+h(t)}{2} \text{\; for all } s,t\in[0,1].
		\]
		Since $h$ is continuous, it follows from Theorem \ref{thm:midpoint-jensen} that $h$ is convex on $[0,1]$.  
		Therefore,
		\[
		h(\lambda)
		\le (1-\lambda)h(0)+\lambda h(1)\text{\; for all } \lambda\in[0,1].
		\]
		Substituting $h=f-g$ gives us
		\[
		\begin{aligned}
			f(\lambda)
			&\le (1-\lambda)f(0)+\lambda f(1)
			-4\alpha\lambda +4\alpha\lambda^2 \\
			&= (1-\lambda)f(0)+\lambda f(1)
			-4\alpha\lambda(1-\lambda),
		\end{aligned}
		\]
		which completes the proof. 
	\end{proof}

The next proposition provides a connection between the strongly midpoint convexity of a  function and its strong quasiconvexity.
	
	\begin{proposition}\label{pro:midpoint-to-lambda}
		Let  $C\subset\mathbb{R}^n$ be a nonempty convex set.
		Consider a continuous function $f\colon C\to\mathbb{R}$.
		Suppose that there exists a constant $\mu \geq 0$ such that 
		\begin{equation}\label{eq:midpoint}
			f\!\left(\frac{x_1+x_2}{2}\right)
			\le 
			\frac{f(x_1)+f(x_2)}{2}
			-\mu\,\|x_1-x_2\|_2^{2}\; \ \mbox{\rm for all }x_1,x_2\in C,
		\end{equation}
		Then $f$ is $\sigma$-strongly quasiconvex on $C$, where $\sigma=8\mu$.
	\end{proposition}
	
	\begin{proof}  Fix any $x_1,x_2 \in C$ with $x_1 \neq x_2$.
		For every $\lambda\in[0,1]$, we define
		\begin{equation}\label{dat1}
			\gamma(\lambda)=\lambda x_1+(1-\lambda)x_2\ \;\text{and}\ \; g(\lambda)=f(\gamma(\lambda)).	    
		\end{equation}
		Then $g$ is continuous on $[0,1]$. For any $s,t\in[0,1]$, we have
		\[
		\gamma\Big(\frac{s+t}{2}\Big)=\frac{\gamma(s)+\gamma(t)}{2}.
		\]
		Applying \eqref{eq:midpoint} with $x_1=\gamma(s)$ and $x_2=\gamma(t)$ gives us
		\[
		g\Big(\frac{s+t}{2}\Big)
		\le \frac{g(s)+g(t)}{2}
		-\mu\,\|\gamma(s)-\gamma(t)\|_2^{2}.
		\]
		Since $\|\gamma(s)-\gamma(t)\|_2=|t-s|\,\|x_1-x_2\|_2,$
		we obtain
		\[
		g\Big(\frac{s+t}{2}\Big)
		\le \frac{g(s)+g(t)}{2}
		-\mu\,(t-s)^2\,\|x_1-x_2\|_2^{2}.
		\]
		Thus, $g$ satisfies  inequality \eqref{eq:g-midpoint} in Lemma~\ref{lem:1D-midpoint-strong} with parameter
		\[
		\alpha=\mu\,\|x_1-x_2\|_2^{2}.
		\]
		By Lemma~\ref{lem:1D-midpoint-strong}, we have
		\begin{equation}\label{dat2}
			g(\lambda)
			\le (1-\lambda)g(0)+\lambda g(1)
			-4\alpha\,\lambda(1-\lambda)\ \;\text{for all }\lambda\in[0,1].    
		\end{equation}
		Since $g(0)=f(x_2)$ and $g(1)=f(x_1)$, it follows from \eqref{dat1} and \eqref{dat2} that
		\[
		f(\gamma(\lambda))
		\le (1-\lambda)f(x_2)+\lambda f(x_1)
		-4\mu\,\lambda(1-\lambda)\|x_1-x_2\|_2^{2}.
		\]
		Setting $\sigma=8\mu$ and utilizing the inequality
		\[
		(1-\lambda)f(x_2)+\lambda f(x_1)
		\le \max\{f(x_1),f(x_2)\},
		\]
		we arrive at
		\[f(\gamma(\lambda))
		\le 
		\max\{f(x_1),f(x_2)\}
		-\frac{\sigma}{2}\,\lambda(1-\lambda)\|x_1-x_2\|_2^{2}.
		\]
		This justifies the 		$\sigma$-strong quasiconvexity of $f$  on $C$.
	\end{proof}

	Next, we establish necessary and sufficient conditions for the strong quasiconvexity of norm functions on a convex set.
	To prove the main result of this section, we rely on two auxiliary lemmas.
	The first lemma may be viewed as a converse to the implication discussed 
	in \cite[Example~2]{jova}.

    \begin{lemma}\label{lem:1Dbounded}
Let $\|\cdot\|$ be an arbitrary norm on $\mathbb{R}^n$, let $x_0,v\in \mathbb{R}^n$ with $\|v\|=1$, and let $I \subset \mathbb{R}$ be a nonempty interval. Define the function $\varphi\colon \mathbb{R}\to \mathbb{R}$ by
\[
\varphi(t)=\|x_0+tv\| \ \; \text{for } t\in\mathbb{R}.
\]
If there exists a constant $\sigma>0$ such that $\varphi$ is $\sigma$-strongly quasiconvex on $I$, then the interval $I$ is necessarily bounded.
\end{lemma}
	
	\begin{proof}
	Suppose that there exists $\sigma>0$ such that $\varphi$ is $\sigma$-strongly quasiconvex on $I$. Suppose on the contrary that  $I$ is unbounded.
		
		\medskip
		\noindent\emph{Step 1: Reduction to the case $0\in I$ and $I$ unbounded above.}
		Fix an element $t_0\in I$. Define  interval $I' = I - t_0$. Obviously, $0 = t_0 - t_0 \in I'$, and $I'$ is also unbounded. We first consider the case where  $I'$ is unbounded from  above. 
		Define $\psi \colon I' \to \mathbb{R}$ by
		\begin{equation}\label{defpsi}
			\psi(s) = \varphi(t_0 + s)\;\text{ for } s \in I'.
		\end{equation}
		We will show that $\psi$ is $\sigma$-strongly quasiconvex on $I'$.
		Indeed, let $s_1,s_2 \in I'$, and let $\lambda \in (0,1)$ be arbitrary.
		Set 
		\begin{equation}\label{deft}
			t_1 = t_0 + s_1\ \;\text{and}\ \; t_2 = t_0 + s_2.    
		\end{equation}
		By the $\sigma$-strong quasiconvexity of $\varphi$ on $I$, we have
		\[
		\varphi\bigl(\lambda t_1 + (1-\lambda)t_2\bigr)
		\le
		\max\{\varphi(t_1),\varphi(t_2)\}
		-\frac{\sigma}{2}\,\lambda(1-\lambda)\,|t_1 - t_2|^2.
		\]
		Using \eqref{defpsi} and \eqref{deft}, the inequality above can be rewritten as
		\[   \psi\bigl(\lambda s_1 + (1-\lambda)s_2\bigr)\le
		\max\{\psi(s_1),\psi(s_2)\}
		-\frac{\sigma}{2}\,\lambda(1-\lambda)\,|s_1 - s_2|^2.
		\]
		Hence, $\psi$ is $\sigma$-strongly quasiconvex on $I'$. Now, we consider the case where 
		$I'$ is unbounded  from below. Consider the  interval 
		$$\widehat I = -I' = \{-s \mid s\in I'\}$$ and define 
		$$\widehat \psi(t) = \varphi(t_0 -t)\;\text{ for }t\in \widehat I.$$
		A direct verification shows that $\widehat\psi$ is also $\sigma$-strongly quasiconvex on $\widehat I$, where the interval $\widehat I$ is  unbounded from above and contains~$0$.
		In both cases, after replacing $I$ and $\varphi$ with one of the pairs $(I',\psi)$ or $(\widehat I,\widehat \psi)$, we may assume without loss of generality that $0\in I$ and $I$ is unbounded from above.
		
		\medskip
		\noindent\emph{Step 2: Contradiction from strong quasiconvexity.}
		Using the $\sigma$-strong quasiconvexity of $\varphi$ with $t_1=0$, 
		$t_2=R\in I$ and $\lambda=1/2$ gives us
		\begin{equation}\label{estimate1}
			\varphi\Big(\tfrac{R}{2}\Big)
			\le\max\{\varphi(0),\varphi(R)\}
			-\frac{\sigma}{2}\cdot\frac14\,R^2
			=\max\{\varphi(0),\varphi(R)\}
			-\frac{\sigma}{8}R^2.    
		\end{equation}
		Since $\|v\|=1$, we see that 
		\begin{equation}\label{estimate2}
			\begin{array}{ll}
				\varphi\Big(\tfrac{R}{2}\Big)&
				=\Big\|x_0+\tfrac{R}{2}v\Big\|\\[0.1in]
				&\ge\tfrac{R}{2}\|v\|-\|x_0\|\\[0.1in]
				&=\tfrac{R}{2}-\|x_0\|.
			\end{array}
		\end{equation}
		For $R$ sufficiently large, we have 
		\begin{equation*}\label{estimate3}
			\begin{array}{ll}
				\max\{\varphi(0),\varphi(R)\}   &=\max\{\|x_0\|,\|x_0+Rv\|\}   \\[0.1in]
				& =\|x_0+Rv\|\\[0.1in]
				&\le R+\|x_0\|.
			\end{array}
		\end{equation*}
		This together with \eqref{estimate1} and \eqref{estimate2} implies that
		\[
		\frac{R}{2}-\|x_0\|
		\le
		\varphi\Big(\tfrac{R}{2}\Big)
		\le
		R+\|x_0\|-\frac{\sigma}{8}R^2.
		\]
		Hence
		\begin{equation*}\label{contractinequality}
			0
			\le
			\frac{1}{2}R+2\|x_0\|-\frac{\sigma}{8}R^2,  
		\end{equation*}
		which leads to a contradiction by letting $R\to \infty$. Therefore, $I$ is  bounded. 
	\end{proof}

	\begin{lemma}\label{lem:ray-in-omega}
		Let $C \subset \mathbb{R}^n$ be a nonempty closed convex set and let $x_0 \in C$.
		Let $v \in \mathbb{R}^n$ be a nonzero vector. Suppose there exist 
		sequences $(v_k)\subset \mathbb{R}^n$ and $(t_k)\subset(0,\infty)$ such that $v_k\to v$, $t_k \to \infty$ and
		$$
		x_k = x_0 + t_k v_k \in C \ \; \text{for all } k \ge 1.
		$$
		Then the entire ray
		$\mathcal{R} = \{x_0 + t v \mid t \ge 0 \}$
		is a subset of $C$.
	\end{lemma}
	
	\begin{proof} Choose $x_k\in C$ such that $x_0+t_kv_k=x_k$ and let $\gamma_k=1/t_k$ for every $k\in \N$. Then $\gamma_k\to 0$ and $\gamma_k(x_k-x_0)=v_k\to v$. By Proposition 6.4 in \cite{mordukhovich2023easy}, we have $$v\in (C-x_0)_\infty,$$ which implies by definition that $\mathcal{R}\subset C$. 		
	\end{proof}
	
	We are now ready to establish the main theorem of this section.

\begin{theorem}\label{thm:Omega}
Consider an arbitrary norm $\|\cdot\|$ on $\mathbb{R}^n$, and let $C \subset \mathbb{R}^n$ be a nonempty convex set. If there exists a constant $\sigma>0$ such that the norm function $f(x)=\|x\|$ is $\sigma$-strongly quasiconvex on $C$, then the set $C$ is necessarily bounded. 
\end{theorem}
	 
	\begin{proof}  
		 Suppose on the contrary that $C$ is unbounded. 
		Since $C$ is convex, its closure $\overline{C}$ is also convex
		and unbounded. By the continuity of $f$, the $\sigma$-strong quasiconvexity of $f$ on $C$
extends to $\overline C$.
		Since $\overline{C}$ is unbounded, there exist $x_0 \in \overline{C}$
		and a sequence $(x_k) \subset \overline{C}$ such that $$\|x_k\|\to\infty\ \; \text{as } k\to \infty.$$
		Setting
		$$
		t_k = \|x_k - x_0\| 
		\ \;\text{and}\ \;
		v_k = \frac{x_k - x_0}{\|x_k - x_0\|},
		$$
		we see that $t_k \to \infty$ and
		$$
		x_k = x_0 + t_k v_k \in \overline{C} \; \mbox{\rm for }k\in \N.
		$$
		Using a subsequence if necessary, we may assume that 
		$v_k \to v$ for some $v \in \mathbb{R}^n$ with $\|v\|=1$.
		Applying Lemma~\ref{lem:ray-in-omega} to the closed convex set 
		$\overline{C}$, we conclude that 
		$\overline{C}$ contains the entire ray
		$$\mathcal{R} = \{x_0 + t v \mid t \ge 0\}.$$
		Define the function $\varphi \colon [0,\infty) \to \mathbb{R}$ by
		$$
		\varphi(t) = \|x_0 + t v\|\;\text{ for } t \in [0,\infty).
		$$
		For $t>s\geq 0$, set
		$$
		x = x_0 + s v 
		\ \;\text{and}\ \; 
		y = x_0 + t v.
		$$
		Then $x,y \in \mathcal{R} \subset \overline{C}$. By the 
		$\sigma$-strong quasiconvexity of $f$ on $\overline{C}$, we have
		\[
		f\big((1-\lambda)x + \lambda y\big)
		\le 
		\max\{f(x),f(y)\}
		-
		\frac{\sigma}{2}\lambda(1-\lambda)\|x-y\|_2^2\;\ \text{for any }\lambda\in(0,1).
		\]
Rewriting the above inequality in terms of $\varphi$, we obtain
\begin{equation}\label{theaboveinequality}
	\varphi\big((1-\lambda)s+\lambda t\big)
	\le 
	\max\{\varphi(s),\varphi(t)\}
	-
	\frac{\sigma}{2}\lambda(1-\lambda)\|v\|_2^2(t-s)^2.
\end{equation}
Since $\|v\|=1$, we have $\|v\|_2>0$. Setting
\[
	\widetilde\sigma=\sigma\|v\|_2^2>0,
\]
inequality~\eqref{theaboveinequality} becomes
\[
	\varphi\big((1-\lambda)s+\lambda t\big)
	\le 
	\max\{\varphi(s),\varphi(t)\}
	-
	\frac{\widetilde\sigma}{2}\lambda(1-\lambda)(t-s)^2.
\]
This shows that $\varphi$ is $\widetilde\sigma$-strongly quasiconvex on the
interval $[0,\infty)$. However, Lemma~\ref{lem:1Dbounded} asserts that
$\varphi$ cannot be $\widetilde\sigma$-strongly quasiconvex on any unbounded
interval. This contradiction shows that $C$ must be bounded, which completes the proof of the theorem.	
	\end{proof}

We end this section with a useful corollary.

\begin{corollary}\label{ref:Omega}
	Let $1<p\le 2$ and equip $\mathbb{R}^n$ with the $\ell_p$-norm $\|\cdot\|_p$.
	Let $C\subset\mathbb{R}^n$ be a nonempty convex set. Then the norm function
	$f(x)=\|x\|_p$ is $\sigma$-strongly quasiconvex on $C$ for some $\sigma>0$
	if and only if $C$ is bounded.
\end{corollary}

\begin{proof}
	The necessity follows directly from Theorem~\ref{thm:Omega}. The sufficiency follows
	from \cite[Example~3.8]{Nam-Jacob}.
\end{proof}

	
	\section{The $\sigma$-strongly quasiconvexity of distance functions} \label{sec4}

    In this section, we establish sufficient conditions under which the distance function to a given set is $\sigma$-strongly quasiconvex. Unless otherwise stated, the distance function is always understood with respect to the norm of the ambient space. To motivate the discussion, we begin with the following illustrative example.

	\begin{example} \label{ex1}
		Consider the Euclidean space $\mathbb{R}^2$ and the set
		$$\Omega = \{(u,v) \in \mathbb{R}^2 \mid u \le 0\}.$$
		Take the point	$x_0 = (1,0) \notin \Omega.$
		In this setting, the Euclidean distance function to $\Omega$ is given by
		$$d_\Omega(u,v) = \max\{0,\,u\}.$$
		Consider a neighborhood of \(x_0\) defined by
		$$U = \{(u,v) \in \mathbb{R}^2\mid \tfrac{1}{2}< u \leq 2 \text{ and } -2 \le v \le 2\}.$$
		Obviously, \(U \cap \Omega = \emptyset\) and $d_\Omega(u,v) = u$ for $u\in U$. 		Choose two points $x_1 = (1,0)$ and $x_2 = (1,1)$
		in \(U\).
		Then we have
		\[
		d_\Omega(x_1) = d_\Omega(x_2) = 1.
		\]
		Setting
		\[
		x_m = \frac{x_1 + x_2}{2} = \big(1,\tfrac{1}{2}\big),
		\]
		we obtain $d_\Omega(x_m) = 1$. 
		Suppose on the contrary that there exists a constant $\sigma>0$ such that $d_\Omega$ is $\sigma$-strongly quasiconvex on~$U$.  
		Then we have
		$$d_\Omega(x_m)
		\le 
		\max\{d_\Omega(x_1), d_\Omega(x_2)\}
		-
		\frac{\sigma}{8}\, \|x_1 - x_2\|_2^2.$$
		Since \(\|x_1 - x_2\|_2 = 1\), the above inequality yields
		$0\leq - \frac{\sigma}{8}$,
		which is a contradiction since $\sigma > 0$.
		Therefore, 
$d_\Omega$ cannot be $\sigma$-strongly quasiconvex on $U$, where $U$ is a bounded neighborhood of $x_0$ such that $U\cap \Omega=\emptyset$. 
	\end{example}
	
	From Example \ref{ex1}, we observe that although the distance function may fail to be $\sigma$-strongly quasiconvex in a neighborhood of a point outside a closed set, it still exhibits a uniform $\sigma$-strong quasiconvexity property along chords joining points on sufficiently large spheres centered at a Euclidean ball. The following proposition formalizes this observation.

\begin{proposition}\label{pro1.4}
	Consider the Euclidean space $\mathbb R^n$ and let $\Omega=\mathbb B^2_R$ be the closed ball centered at the origin with radius $R>0$. Then for every $r>R$, there exists a constant $\sigma_0>0$ such that for all
		$x_1,x_2\in\mathbb S^2_r$, one has
		\begin{equation}\label{key5}
			d_{\Omega}\big((1-\lambda)x_1+\lambda x_2\big)
			\le
			\max\{d_{\Omega}(x_1),d_{\Omega}(x_2)\}
			-
			\frac{\sigma_0}{2}\lambda(1-\lambda)\|x_1-x_2\|_2^2
		\end{equation}
		for all $\lambda\in(0,1)$. 
		If in addition the segment $[x_1,x_2]$ satisfies
		\[
			[x_1,x_2]\cap\mathbb B^2_R=\emptyset,
		\]
		then there exists a constant $\sigma_1>0$ such that the distance function $d_\Omega$ is $\sigma_1$-strongly quasiconvex on $[x_1,x_2]$.
\end{proposition}

\begin{proof}
	 Fix any $r>R$ and take arbitrary $x_1,x_2\in\mathbb S^2_r$.
	Since $\|x_1\|_2=\|x_2\|_2=r$, we have
	\[
		d_\Omega(x_1)=d_\Omega(x_2)=r-R.
	\]
	Given $\lambda\in(0,1)$, define
	\[
		x_\lambda=(1-\lambda)x_1+\lambda x_2.
	\]
Observe that
	\[
		d_\Omega(x_\lambda)=\max\{\|x_\lambda\|_2-R,0\}.
	\]
	Therefore,
	\begin{equation}\label{eq:cases}
		d_\Omega(x_1)-d_\Omega(x_\lambda)
		=
		\begin{cases}
			r-\|x_\lambda\|_2 & \text{if } \|x_\lambda\|_2\ge R,\\[1mm]
			r-R & \text{if } \|x_\lambda\|_2<R.
		\end{cases}
	\end{equation}
	Let $\theta\in[0,\pi]$ be the angle between $x_1$ and $x_2$. Then
	\begin{equation}\label{eq:key1}
		\|x_1-x_2\|_2^2
		= \|x_1\|_2^2 + \|x_2\|_2^2 - 2\langle x_1,x_2\rangle=
		2r^2(1-\cos\theta).
	\end{equation}
	Moreover,
		\begin{align*}
			\|x_\lambda\|_2^2
			&= \|(1-\lambda)x_1 + \lambda x_2\|_2^2 \\
			&= (1-\lambda)^2\|x_1\|_2^2 + \lambda^2\|x_2\|_2^2 
			+ 2\lambda(1-\lambda)\langle x_1,x_2\rangle \\
			&= r^2\bigl((1-\lambda)^2 + \lambda^2 + 2\lambda(1-\lambda)\cos\theta\bigr) \\
			&= r^2\Big(1 - 2\lambda + 2\lambda^2 + 2\lambda(1-\lambda)\cos\theta\Big).
		\end{align*}
Setting	
		\begin{equation}\label{keyA}
			A = 1 - 2\lambda + 2\lambda^2 + 2\lambda(1-\lambda)\cos\theta,
		\end{equation}
		we get 
		\begin{equation}\label{key11}
			\|x_\lambda\|_2 = r\sqrt{A}.
		\end{equation}
		Since $\cos\theta\in[-1,1]$, we have
		\[
		(1-2\lambda)^2 \le A \le 1 - 2\lambda + 2\lambda^2 + 2\lambda(1-\lambda) = 1,
		\]
		and hence $A\in[0,1]$.
		It follows from the concavity of the square root function on $[0,1]$ that
		\begin{equation}\label{keycanA}
			1-\sqrt{A} \;\ge\; \frac{1-A}{2}.
		\end{equation}
	Combining \eqref{eq:key1} through \eqref{keycanA} gives us
		\begin{equation}\label{key2}
			\begin{array}{ll}
				r-\|x_\lambda\|_2
				&=r(1-\sqrt{A})\\[2mm] 
				&\ge \frac{r}{2}\,(1-A)\\[2mm]
				&\ge r\lambda(1-\lambda)(1-\cos\theta)\\[2mm]
				&=\frac{1}{2r}\,\lambda(1-\lambda)\|x_1-x_2\|_2^2.
			\end{array}
		\end{equation}
		Since $0<\lambda(1-\lambda)\le\tfrac14$ and 
		$\|x_1-x_2\|_2\le 2r$, we obtain
		\[
		\lambda(1-\lambda)\|x_1-x_2\|_2^2 \le r^2,
		\]
		and hence
		\begin{equation}
			r-R \;\ge\; \frac{r-R}{r^2}\,\lambda(1-\lambda)\|x_1-x_2\|_2^2.
			\label{eq:lower2}	
		\end{equation}
		Using \eqref{eq:cases}, \eqref{key2} and \eqref{eq:lower2}, we get the estimate
		\[
		d_\Omega(x_1)-d_\Omega(x_\lambda)
		\;\ge\;
		\min\!\left\{\frac{1}{2r},\,\frac{r-R}{r^2}\right\}
		\lambda(1-\lambda)\|x_1-x_2\|_2^2.
		\]
		Choosing
		\[\sigma_0 = 2\min\!\left\{\frac{1}{2r},\,\frac{r-R}{r^2}\right\},\]	
		we arrive at
		\[d_\Omega(x_1)-d_\Omega(x_\lambda)\ge \frac{\sigma_0}{2}\lambda(1-\lambda)\|x_1-x_2\|_2^2.\]
		Since $d_\Omega(x_1)=d_\Omega(x_2)$, this inequality directly yields \eqref{key5}.
	
	\noindent Under the additional assumption imposed, let us show next that $d_\Omega$ is $\sigma_1$-strongly quasiconvex on $[x_1,x_2]$.
	By Corollary~\ref{ref:Omega}, applied to the Euclidean norm on the bounded convex
set $C=[x_1,x_2]$, there exists $\sigma_1>0$ such that the Euclidean norm
function $y\mapsto\|y\|_2$ is $\sigma_1$-strongly quasiconvex on $[x_1,x_2]$.
 Then it follows from Definition~\ref{defquasiconvex} that
	\[
		\|(1-\lambda)y_1+\lambda y_2\|_2
		\le
		\max\{\|y_1\|_2,\|y_2\|_2\}
		-
		\frac{\sigma_1}{2}\lambda(1-\lambda)\|y_1-y_2\|_2^2,
	\]
    for all
	$y_1,y_2\in [x_1,x_2]$ and all $\lambda\in(0,1).$
	Since $d_\Omega(y)=\|y\|_2-R$ for all $y\in [x_1,x_2]$, subtracting~$R$ from both sides gives us  
    \begin{equation*}
			d_{\Omega}\big((1-\lambda)y_1+\lambda y_2\big)
			\le
			\max\{d_{\Omega}(y_1),d_{\Omega}(y_2)\}
			-
			\frac{\sigma_1}{2}\lambda(1-\lambda)\|y_1-y_2\|_2^2
		\end{equation*}
		for all $y_1,y_2\in[x_1,x_2]$ and all $\lambda\in(0,1)$,
    which therefore completes the proof
of the proposition.
\end{proof}

	\begin{remark} Inequality \eqref{key5} in		Proposition~\ref{pro1.4} relies on the Euclidean structure of the norm.
Indeed, the conclusion may fail for a non-Euclidean norm on $\mathbb{R}^n$ as illustrated by the following example.
	\end{remark}
	
	\begin{example} Consider the plane $\mathbb{R}^2$ with the  maximum norm $\|\cdot\|_\infty$. Let $R=1/2$ and $r=2$. Then the closed ball and the sphere centered $0$ with radius $R$ and $r$ are given by
		$$\Omega=\mathbb{B}^\infty_R=\{(u,v)\in\mathbb{R}^2\mid \|(u,v)\|_\infty=\max\{|u|,|v|\}\leq 1/2\}=[-1/2,1/2]^2$$
		and
		$$\mathbb{B}^\infty_r=\{(u,v)\in \mathbb R^2\mid \max\{|u|,|v|\}\leq 2\}=[-2,2]^2,$$
		respectively.
		Take $x_1=(2,0)$ and $x_2=(2,1)$, which belong to $\mathbb{S}^\infty_r$. Then we have
		$$\|x_1\|_\infty=\|x_2\|_\infty=2.$$
		For any $\lambda\in (0,1)$, we have
		$$\|x_\lambda\|_\infty=\|(1-\lambda)x_1+\lambda x_2\|_\infty=\max\{2,\lambda\}=2.$$
		The distances from all these three points to $\Omega$ with respect to $\|\cdot\|_\infty$ are computed by
		$$d_\Omega(x_1)=d_\Omega(x_2)=d_\Omega(x_\lambda)=r-R=2-1/2=3/2\ \;\text{for all } \lambda\in (0,1).$$
		Suppose on the contrary that there exists a constant $\sigma >0$ satisfying 
		$$d_\Omega(x_\lambda)\leq \max\{d_\Omega(x_1),d_\Omega(x_2)\}-\frac{\sigma}{2}\, \lambda(1-\lambda)\, \|x_1 - x_2\|_2^2\; \ \mbox{\rm for all } \lambda\in(0,1).$$
		Since $\|x_1-x_2\|_2 = \|(0,-1)\|_2 = 1$, this inequality yields
		\[
		\tfrac32
		\le 
		\tfrac32
		-
		\frac{\sigma}{2}\,\lambda(1-\lambda),
		\]
		or equivalently,
		\[
		0 \le -\,\frac{\sigma}{2}\,\lambda(1-\lambda).
		\]
		This is a contradiction because $\sigma>0$ and $\lambda\in(0,1)$.
		Therefore, inequality \eqref{key5} fails for the maximum norm. 
	\end{example}
	
	\begin{remark}
		Inequality \eqref{key5} does not hold for all pairs $x_1,x_2\in\mathbb{R}^n$.
		Indeed, take any two distinct points $x_1,x_2\in\Omega=\mathbb{B}^2_R$.
		Since $\Omega$ is convex, the whole segment
		\[
		[x_1,x_2]
		= \{(1-\lambda)x_1+\lambda x_2 \mid \lambda\in[0,1]\}
		\]
		is contained in $\Omega$. Thus,
		\[
		d_\Omega(x_1) = d_\Omega(x_2) = 0
		\ \;\text{and}\ \;
		d_\Omega\big((1-\lambda)x_1+\lambda x_2\big) = 0
		\ \;\text{for all }\lambda\in(0,1).
		\]
		If there exists $\sigma>0$ such that \eqref{key5} holds for this pair $x_1,x_2$,
		then we obtain
		\begin{align*}
			0
			= d_\Omega\big((1-\lambda)x_1+\lambda x_2\big)
			&\le
			\max\{d_\Omega(x_1),d_\Omega(x_2)\}
			- \frac{\sigma}{2}\,\lambda(1-\lambda)\,\|x_1-x_2\|_2^2 \\
			&= - \frac{\sigma}{2}\,\lambda(1-\lambda)\,\|x_1-x_2\|_2^2
		\end{align*}
		for all $\lambda\in(0,1)$. Since $x_1\neq x_2$ and $\lambda(1-\lambda)>0$, the
		right-hand side is strictly negative, which is not the case.
	\end{remark}

The next result shows that the distance function $d_\Omega$ also possesses a
\textit{local} $\sigma$-strong quasiconvexity around every point lying outside
$\Omega$.

\begin{theorem}\label{pro1.3}
	Consider the space $\mathbb R^n$ equipped with the $\ell_p$-norm
	$\|\cdot\|_p$, where $1<p\le 2$. Let
		$\Omega=\mathbb B_R^p$
	be the closed ball centered at the origin with radius $R>0$. Then, for every
	point $x_0\notin\Omega$, there exist a bounded neighborhood $U$ of $x_0$ with
	$U\cap\Omega=\emptyset$ and a constant $\sigma>0$ such that the distance
	function $d_\Omega$ is $\sigma$-strongly quasiconvex on $U$.
\end{theorem}

\begin{proof}
	Since $\Omega$ is closed and $x_0\notin\Omega$, we have
	\[
		\delta=d_\Omega(x_0)=\|x_0\|_p-R>0.
	\]
	Choose $\varepsilon\in(0,\delta)$ and define
	\[
		U=\{x\in\mathbb R^n\mid \|x-x_0\|_p<\varepsilon\}.
	\]
	Then $U$ is a bounded convex neighborhood of $x_0$. Moreover, for every
	$x\in U$, we have
	\[
		\|x\|_p\ge \|x_0\|_p-\|x-x_0\|_p>\|x_0\|_p-\varepsilon>R.
	\]
	Hence $U\cap\Omega=\emptyset$.
	Since $U$ is a nonempty bounded convex subset of $(\mathbb R^n,\|\cdot\|_p)$
	with $1<p\le2$, Corollary~\ref{ref:Omega} guarantees the existence of a constant
	$\sigma>0$ such that the norm function $x\mapsto \|x\|_p$ is
	$\sigma$-strongly quasiconvex on $U$, i.e.,
	\[
		\|(1-\lambda)x+\lambda y\|_p
		\le
		\max\{\|x\|_p,\|y\|_p\}
		-
		\frac{\sigma}{2}\lambda(1-\lambda)\|x-y\|_2^2
	\]
	for all $x,y\in U$ and all $\lambda\in(0,1)$. Since
	\[
		d_\Omega(x)=\|x\|_p-R\ \;\text{for all }x\in U,
	\]
	we obtain, by subtracting $R$ from both sides, that
	\[
	\begin{aligned}
		d_\Omega\big((1-\lambda)x+\lambda y\big)
		&=\|(1-\lambda)x+\lambda y\|_p-R\\
		&\le
		\max\{\|x\|_p,\|y\|_p\}-R
		-
		\frac{\sigma}{2}\lambda(1-\lambda)\|x-y\|_2^2\\
		&=
		\max\{\|x\|_p-R,\|y\|_p-R\}
		-
		\frac{\sigma}{2}\lambda(1-\lambda)\|x-y\|_2^2\\
		&=
		\max\{d_\Omega(x),d_\Omega(y)\}
		-
		\frac{\sigma}{2}\lambda(1-\lambda)\|x-y\|_2^2.
	\end{aligned}
	\]
	This shows that $d_\Omega$ is $\sigma$-strongly quasiconvex on $U$, and the
	proof is complete.
\end{proof}

 Theorem~\ref{pro1.3} involves the $\ell_p$-norm with $1<p\le2$ and does not
apply to an arbitrary norm, as shown in the following example.

\begin{example}
	Consider the plane $\mathbb R^2$ equipped with the $\ell_1$-norm. Let $R=1$
	and consider the $\ell_1$-ball
	\[\Omega=\mathbb B_R^1
		=
		\{(u,v)\in\mathbb R^2\mid \|(u,v)\|_1=|u|+|v|\le1\}.\]
	Take the point $x_0=(2,0)$, which lies outside $\Omega$. The $\ell_1$-distance
	to $\Omega$ is given by
	\[d_\Omega(u,v)=\max\{|u|+|v|-1,0\}.\]
	Suppose that $U$ is a bounded neighborhood of $x_0=(2,0)$ such that
	$U\cap\Omega=\emptyset$. Then there exists $\rho>0$ such that
	\[
		\mathbb B^1[(2,0);\rho]\subset U.
	\]
	Choose $\varepsilon\in(0,\rho)$ and take
	\[
		x_1=(2,0)
		\ \;\text{and}\ \;
		x_2=\left(2-\frac{\varepsilon}{2},\frac{\varepsilon}{2}\right).
	\]
	Then $x_1,x_2\in U$. Moreover,
	\[
		d_\Omega(x_1)=d_\Omega(x_2)=1.
	\]
	Their midpoint is
	\[
		x_m=\frac{x_1+x_2}{2}
		=
		\left(2-\frac{\varepsilon}{4},\frac{\varepsilon}{4}\right),
	\]
	and hence
	\[d_\Omega(x_m)=1.\]
	Suppose on the contrary that there exists a constant $\sigma>0$ such that
	$d_\Omega$ is $\sigma$-strongly quasiconvex on $U$. Then
	\[
		d_\Omega(x_m)
		\le
		\max\{d_\Omega(x_1),d_\Omega(x_2)\}
		-
		\frac{\sigma}{2}\cdot\frac14\|x_1-x_2\|_2^2.
	\]
	Since
	\[
		\|x_1-x_2\|_2^2
		=
		\left(\frac{\varepsilon}{2}\right)^2
		+
		\left(\frac{\varepsilon}{2}\right)^2
		=
		\frac{\varepsilon^2}{2},
	\]
	the above inequality yields $0\le -\frac{\sigma\varepsilon^2}{16},$
	which is impossible since $\sigma>0$ and $\varepsilon>0$. This shows that Theorem~\ref{pro1.3} fails for the
	$\ell_1$-norm.
\end{example}
    

	Next, we consider the concept of a \textit{strongly convex set}. This notion was  proposed in \cite{vial} and further studied in \cite{Martinez,nacry,Zalinescu}.

	Given $R>0$ and $x,y \in \mathbb{R}^n$ with
	$\|x-y\|_2 \le 2R$, define the set
	\[
	D_R(x,y)
	= \bigcap \{ B \in \mathcal{B}_R \mid x,y \in B \},
	\]
	where $\mathcal{B}_R = \{\mathbb{B}^2[c,R] \mid c \in \mathbb{R}^n\}$ is the collection of all closed balls of radius $R$.

	\begin{definition}\label{stronglyconvex}
		Let $\Omega$ be a nonempty bounded subset of the Euclidean space $\mathbb{R}^n$.
		We say that $\Omega$ is \emph{strongly convex with respect to} a positive real number $R$ if
		\[
		D_R(x,y) \subset \Omega
		\ \; \text{for all } x,y \in \Omega.
		\]
	\end{definition}
	
	The next theorem extends the result of Proposition \ref{pro1.4} to the case of strongly convex sets.

	\begin{theorem}\label{pro1.9-correct}
		Consider the Euclidean space $\mathbb{R}^n$.  
		Let $\Omega\subset \mathbb{R}^n$ be a nonempty compact set which is
		strongly convex with respect to $R>0$.
		Let $r>0$ be such that $$\mathbb{S}^2_r\cap\Omega=\emptyset.$$
		Suppose  that there exists a constant $c_0>0$ such that  
		\begin{equation}\label{metric-regular-sphere}
			\|p_1-p_2\|_2 \;\ge\; c_0\,\|x_1-x_2\|_2\; \ \mbox{\rm for all
		}x_1,x_2\in\mathbb{S}^2_r,
		\end{equation}
        where $p_i$ denotes the Euclidean projection
		of $x_i$ onto $\Omega$, $i=1,2$.
		Then there exists $\sigma>0$ (one may take $\sigma=c^2/R$ for some
		$0<c\le c_0$) such that for all $x_1,x_2\in\mathbb{S}^2_r$, one has
		\begin{equation}\label{key5-correct}
			d_{\Omega}\big( (1-\lambda) x_1 + \lambda x_2 \big)
			\le
			\max\{ d_{\Omega}(x_1),\, d_{\Omega}(x_2) \}
			-
			\frac{\sigma}{2}\,\lambda(1-\lambda)\,\|x_1 - x_2\|_2^2
		\end{equation}
         for all
		$\lambda\in(0,1)$.
	\end{theorem}
	
	\begin{proof}
		Since $\Omega$ and $\mathbb{S}^2_r$ are disjoint compact sets, the distance between them is attained and strictly positive. Define
		\begin{equation}\label{eq:delta-lower}
			\delta = \min\{\|y-x\|_2\mid y\in\mathbb{S}^2_r,\ x\in\Omega\} > 0.
		\end{equation}
		Observe that if \eqref{metric-regular-sphere} holds for some $c_0>0$,
		then it also holds for every constant $c$ with $0<c\le c_0$.
		Choose
		\begin{equation}\label{eq:c-choice}
			c = \min\Bigl\{c_0,\sqrt{\frac{2R\delta}{r^2}}\Bigr\} > 0,
		\end{equation}
		and define
		\begin{equation}\label{eq:sigma-def}
			\sigma = \frac{c^2}{R} > 0.
		\end{equation}
		Then \eqref{metric-regular-sphere} holds with this $c$ and we have
		\begin{equation}\label{eq:sigma-delta}
			\frac{\sigma}{2}\,r^2
			= \frac{c^2}{2R}\,r^2
			\;\le\; \delta.
		\end{equation}
		Now fix any $x_1,x_2\in\mathbb{S}^2_r$. Since $\Omega$ is strongly convex with respect to $R>0$, it is also
		convex. Hence, the Euclidean projection $P_\Omega$ is single-valued
		(see, e.g., \cite[Corollary~1.67]{mordukhovich2023easy}).
		For $i=1,2$, let $p_i=P_\Omega(x_i)$ 
		be the Euclidean projection of $x_i$ onto~$\Omega$. 
		For each $\lambda\in(0,1)$, set
		\[
		x_\lambda = (1-\lambda)x_1 + \lambda x_2
		\ \;\text{and}\ \;
		p_\lambda = (1-\lambda)p_1 + \lambda p_2.
		\]
		Since $\Omega$ is strongly convex with respect to $R$, it follows from the equivalence between (i) and (ii) 
		in \cite[Theorem~1]{vial} that 
		\[
		\mathbb{B}^2\big[p_\lambda;r_\lambda\big] \subset \Omega \;\;\text{for any }\lambda\in(0,1),
			\]
		where
		$$r_\lambda = \frac{\lambda(1-\lambda)\,\|p_1-p_2\|_2^2}{2R}.$$
	Thus,
		\begin{equation}\label{eq:dist-ball}
			d_\Omega(x_\lambda)
			\;\le\;
			d_{\mathbb{B}^2[p_\lambda;r_\lambda]}(x_\lambda)
			=
			\max\bigl\{\|x_\lambda-p_\lambda\|_2-r_\lambda,\,0\bigr\}.
		\end{equation}
		By the triangle inequality, we have
		\begin{equation}\label{eq:m-lambda-p-lambda}
			\begin{array}{ll}
				\|x_\lambda-p_\lambda\|_2
				& \le (1-\lambda)\|x_1-p_1\|_2 + \lambda\|x_2-p_2\|_2\\ [0.1 in]
				& = (1-\lambda)d_\Omega(x_1) + \lambda d_\Omega(x_2).
			\end{array}
		\end{equation}
		Furthermore, by \eqref{metric-regular-sphere} and the choice of $c$ in
\eqref{eq:c-choice}, we have
		$$\|p_1-p_2\|_2\ge c\|x_1-x_2\|_2.$$ 
		This together with \eqref{eq:sigma-def} implies
		\begin{equation}\label{eq:r-lambda-lower}
			\begin{array}{ll}
				r_\lambda
				&= \frac{\lambda(1-\lambda)\,\|p_1-p_2\|_2^2}{2R}\\[0.1in]
				&\ge  \frac{\lambda(1-\lambda)c^2}{2R}\,\|x_1-x_2\|_2^2  \\[0.1in]
				& =\frac{\sigma}{2}\,\lambda(1-\lambda)\,\|x_1-x_2\|_2^2.
			\end{array}
		\end{equation}
		Next, we consider two cases: $x_\lambda\notin\Omega$ or $x_\lambda\in\Omega$. 
		In the case where $x_\lambda\notin\Omega$, we have
		\[
		d_{\mathbb{B}^2[p_\lambda;r_\lambda]}(x_\lambda)
		= \|x_\lambda-p_\lambda\|_2-r_\lambda.
		\]
		Using this equality with \eqref{eq:dist-ball}, \eqref{eq:m-lambda-p-lambda} and
		\eqref{eq:r-lambda-lower}, we obtain
		\begin{align*}
			d_\Omega(x_\lambda)
			&\le \|x_\lambda-p_\lambda\|_2 - r_\lambda \\
			&\le (1-\lambda)d_\Omega(x_1) + \lambda d_\Omega(x_2) - \frac{\sigma}{2}\,\lambda(1-\lambda)\,\|x_1-x_2\|_2^2 \\
			&\le \max\{d_\Omega(x_1),d_\Omega(x_2)\}
			- \frac{\sigma}{2}\,\lambda(1-\lambda)\,\|x_1-x_2\|_2^2.
		\end{align*}
		This justifies \eqref{key5-correct}. 
		If $x_\lambda\in\Omega$, then $d_\Omega(x_\lambda)=0$, and 
		it follows from~\eqref{eq:delta-lower} that 
		\begin{equation}\label{key18}
			\delta\le \max\{d_\Omega(x_1),d_\Omega(x_2)\}.
		\end{equation}
		Since $x_1,x_2\in\mathbb{S}^2_r$, we have $$\|x_1-x_2\|_2\le 2r\ \;\text{and}\ \;
		0<\lambda(1-\lambda)\le\frac14,$$ 
		which implies that
		\[
		\frac{\sigma}{2}\,\lambda(1-\lambda)\,\|x_1-x_2\|_2^2
		\le
		\frac{\sigma}{2}\,r^2.
		\]
		Combining this with \eqref{eq:sigma-delta} and \eqref{key18}, we obtain
		\[
		\frac{\sigma}{2}\,\lambda(1-\lambda)\,\|x_1-x_2\|_2^2
		\;\le\;\frac{\sigma}{2}\,r^2  \;\le\; \delta
		\;\le\; \max\{d_\Omega(x_1),d_\Omega(x_2)\},
		\]
		which shows that \eqref{key5-correct} also holds in this case.
        Combining the two cases, we conclude that for all $x_1,x_2\in\mathbb S^2_r$,
inequality~\eqref{key5-correct} holds for every $\lambda\in(0,1)$. This completes
the proof.
	\end{proof}

\begin{remark}\label{lem:MR-ball}
	Theorem~\ref{pro1.9-correct} may be viewed as an extension of
	Proposition~\ref{pro1.4} from a Euclidean ball to an arbitrary
	strongly convex set. Indeed, let
	\[
		\Omega=\mathbb B^2[0;R]\subset\mathbb R^n
	\]
	be the Euclidean ball of radius $R>0$. Fix $r>R$ and take arbitrary
	points $x_1,x_2\in\mathbb S^2_r$. For $i=1,2$, let $p_i$ denote the
	Euclidean projection of $x_i$ onto $\Omega$. Since the projection onto a
	Euclidean ball is radial, we have
	\[
		p_i=\frac{R}{r}x_i,\qquad i=1,2.
	\]
	Consequently,
	\[
		\|p_1-p_2\|_2
		=
		\frac{R}{r}\|x_1-x_2\|_2.
	\]
	Thus, condition~\eqref{metric-regular-sphere} holds with $c_0=R/r$.
\end{remark}

The following example shows that condition~\eqref{metric-regular-sphere} cannot
be extended to arbitrary neighborhoods of points lying outside a ball.

\begin{example}\label{ex:MR-fails-neighborhood}
	Consider the Euclidean space $\mathbb R^n$, and let $\Omega=\mathbb B^2[0;1]$
	be the closed unit ball. Fix any point $x_0\notin\Omega$ and let
	\[
		\delta=d_\Omega(x_0)=\|x_0\|_2-1>0.
	\]
	Let $U$ be any neighborhood of $x_0$ satisfying
		$U\cap\Omega=\emptyset.$
	Then there exists $\rho>0$ such that
	\[
		\mathbb B^2[x_0;\rho]\subset U.
	\]
	Choose $\varepsilon\in(0,\min\{\rho,\delta\})$. Set	$u=\frac{x_0}{\|x_0\|_2}$
	and define
	\[
		x_1=x_0-\varepsilon u
		=(\|x_0\|_2-\varepsilon)u,
		\qquad
		x_2=x_0+\varepsilon u
		=(\|x_0\|_2+\varepsilon)u.
	\]
	Then
	\[
		x_1,x_2\in\mathbb B^2[x_0;\rho]\subset U.
	\]
	Since $\varepsilon<\delta=\|x_0\|_2-1$, we see that both $x_1, x_2\notin \Omega$. Note that their Euclidean projections onto $\Omega$ coincide with 
	\[
		P_\Omega(x_1)=P_\Omega(x_2)=u.
	\]
	Consequently,
	\[
		\|P_\Omega(x_1)-P_\Omega(x_2)\|_2=0.
	\]
	On the other hand,
	\[
		\|x_1-x_2\|_2=2\varepsilon>0.
	\]
	Thus, for any constant $c>0$, one has
	\[
		\|P_\Omega(x_1)-P_\Omega(x_2)\|_2
		=
		0
		<
		c\|x_1-x_2\|_2.
	\]
	This shows that an estimate of the form~\eqref{metric-regular-sphere} cannot
	hold uniformly on arbitrary neighborhoods of points outside $\Omega$. 
\end{example}

	\section{Conclusions}\label{sec5}
	
	This paper provides further insights into the strong quasiconvexity of norm functions while initiating the study of the strong quasiconvexity of distance functions in finite-dimensional spaces. These results open several promising directions for future research. These include sharpening convergence rates for projection-type algorithms, investigating generalized distance functions in variational analysis and applications to economics and data science, and extending the present theory to infinite-dimensional settings.


\begin{thebibliography}{99}

\bibitem{Bauschke}
Bauschke HH, Combettes PL. Convex analysis and monotone operator theory in Hilbert spaces. 2nd ed. New York: Springer; 2017.

\bibitem{Clarkson}
Clarkson JA. Uniformly convex spaces. Trans Amer Math Soc. 1936;40(3):396--414.

\bibitem{Crouzeix}
Crouzeix JP. Continuity and differentiability of quasiconvex functions. In: Hadjisavvas N, Koml\'{o}si S, Schaible S, editors. Handbook of generalized convexity and generalized monotonicity. Nonconvex optimization and its applications. Vol. 76. New York: Springer; 2005. p. 121--139.

\bibitem{Grad}
Grad SM, Lara F, Marcavillaca RT. Strongly quasiconvex functions: what we know (so far). J Optim Theory Appl. 2025;205(2):38. doi:10.1007/s10957-025-02641-4.

\bibitem{Hadjisavvas}
Hadjisavvas N. Convexity, generalized convexity and applications. In: Al-Mezel S, et al., editors. Fixed point theory, variational analysis and optimization. Boca Raton: Taylor \& Francis; 2014. p. 139--169.


\bibitem{Hadjisavvas1}
Hadjisavvas N, Lara F. Characterizations of strongly quasiconvex functions. J Global Optimi. 2026;94(4);997--1003.

\bibitem{Hadjisavvas2}
	Hadjisavvas N, Lara F,  Martínez-Legaz, JE. A quasiconvex asymptotic function with applications in optimization. J Optim Theory and Appl. 2019;180(1):170--186.
	
\bibitem{jensen}
Jensen JLWV. Om konvekse funktioner og uligheder imellem middelvaerdier. Nyt Tidsskr Math. 1905;16:49--68.

\bibitem{jova}
Jovanovi\'{c} MV. A note on strongly convex and quasiconvex functions. Math Notes. 1996;60(5):584--585.


\bibitem{lara2}
Lara F. On strongly quasiconvex functions: existence results and proximal point algorithms. J Optim Theory Appl. 2022;192:891--911.

\bibitem{lara3}
Lara F, Marcavillaca RT, Vuong PT. Characterizations, dynamical systems and gradient methods for strongly quasiconvex functions. J Optim Theory Appl. 2024;206:60.

\bibitem{lara1} Lara F, Marcavillaca RT, Yen LH. An extragradient projection method for strongly quasiconvex equilibrium problems with applications. Comput. Appl. Math.2024;43(3):128.

\bibitem{Martinez}
Mart\'{i}nez-Legaz JE. On strongly convex sets and farthest distance functions. arXiv preprint arXiv:2507.15053; 2025.

\bibitem{mordukhovich2023easy}
Mordukhovich BS, Nam NM. An easy path to convex analysis and applications. 2nd ed. Cham: Springer; 2023.

\bibitem{nacry}
Nacry F, Nguyen VAT, Thibault L. Farthest distance function to strongly convex sets. J Convex Anal. 2023;30:1217--1240.

\bibitem{Nam-Jacob}
Nam NM, Sharkansky J. On strong quasiconvexity of functions in infinite dimensions. Optim Lett. 2025;19(6):1849--1865. doi:10.1007/s11590-025-02261-x.

\bibitem{Nesterov}
Nesterov Y. Introductory lectures on convex optimization: a basic course. Boston: Springer; 2004.

\bibitem{polyak}
Polyak BT. Existence theorems and convergence of minimizing sequences in extremum problems with restrictions. Sov Math Dokl. 1966;7:72--75.

\bibitem{rock27}
Rockafellar RT. Convex analysis. Princeton (NJ): Princeton University Press; 1970.

\bibitem{vanhang}
Van Hang NT, Lara F, Yen ND. Discontinuous Strongly Quasiconvex Functions. J Optim Theory Appl. 2026;210:40. doi.org/10.1007/s10957-026-03066-3

\bibitem{vial}
Vial JP. Strong convexity of sets and functions. J Math Econ. 1982;9(1--2):187--205.

\bibitem{Zalinescu}
Z\u{a}linescu C. On uniformly convex functions. J Math Anal Appl. 1983;95(2):344--374.

\bibitem{za28}
Z\u{a}linescu C. Convex analysis in general vector spaces. Singapore: World Scientific; 2002.

\end{thebibliography}
\end{document}